\documentclass{amsart}
\usepackage{amsmath,amsthm,amssymb}
\usepackage{graphicx}
\usepackage{hyperref} 
\hypersetup{
	setpagesize=false,
	bookmarksnumbered=true,%
	bookmarksopen=true,%
	colorlinks=true,%
	linkcolor=magenta,
	citecolor=blue,
}
\theoremstyle{plain}
\newtheorem{thm}{Theorem}[section]
\newtheorem{lem}[thm]{Lemma}
\newtheorem{prop}[thm]{Proposition}
\newtheorem{cor}[thm]{Corollary}
\newtheorem{conj}[thm]{Conjecture}
\theoremstyle{definition}
\newtheorem{defi}[thm]{Definition}
\newtheorem{exm}[thm]{Example}
\newtheorem{rem}[thm]{Remark}
\newtheorem{qst}[thm]{Question}
\newcommand{\ZZ}{\mathbb{Z}}
\newcommand{\CC}{\mathbb{C}}
\newcommand{\calD}{\mathcal{D}}
\newcommand{\Int}{\mathrm{Int\,}}
\title[Exotic 4-manifolds with small trisection genus]{Exotic 4-manifolds with small trisection genus}
\author[Natsuya Takahashi]{Natsuya Takahashi}
\date{December 7, 2024.}
\address{Department of Pure and Applied Mathematics, Graduate School of Information Science and Technology, Osaka University, 1-5 Yamadaoka, Suita, Osaka 565-0871, Japan}
\email{nt-takahashi@ist.osaka-u.ac.jp}
\begin{document}
\begin{abstract}
We show that there exists an exotic pair of $4$-manifolds with boundary whose trisection genera are $4$. We also construct genus-$3$ relative trisections for an infinite family of contractible $4$-manifolds introduced by Akbulut and Kirby.
\end{abstract}
\maketitle
\section{Introduction}\label{sec:intro}

 The theory of \textit{trisections} for smooth $4$-manifolds was introduced by Gay and Kirby~\cite{GayKir16} as a $4$-dimensional analogue of Heegaard splittings for $3$-manifolds.
 Roughly speaking, a trisection is a decomposition of a smooth $4$-manifold into three $4$-dimensional $1$-handlebodies which intersect along a surface.
 The main focus of this paper is on trisections for compact $4$-manifolds with non-empty connected boundary, called \textit{relative trisections}. 
 (For foundations of relative trisections, see e.g., \cite{GayKir16}, \cite{Cas16}, \cite{CasGayPin18_1}, \cite{CasGayPin18_2}, \cite{CasOzb19}, \cite{KimMil20}.)
 Gay and Kirby also showed that any compact, connected, oriented, smooth $4$-manifold admits a (relative) trisection, and smooth structures of $4$-manifolds can be encoded by trisection diagrams.

 \textit{Trisection genus} is a natural $4$-dimensional analogue of Heegaard genus for $3$-manifolds.
  For a smooth $4$-manifold $X$, the trisection genus $g(X)$ is defined as the minimal integer $g$ such that $X$ admits a (relative) trisection with the triple intersection surface of genus $g$.
 An important fact in $3$-dimensional topology is that Heegaard genus is additive under connected sum, which follows from Haken's lemma~\cite{Hak68}.
 As a $4$-dimensional analogue of this result, Lambert-Cole and Meier~\cite{LamMei20} conjectured that trisection genus is additive under connected sum, that is, for any two $4$-manifolds $X$ and $Y$, we have $g(X\# Y)=g(X)+g(Y)$.
 (While they focused on closed $4$-manifolds, we can also consider the relative version: trisection genus for $4$-manifolds with boundary is additive under boundary connected sum.)
 It is worth noting that if the above conjectures hold, then trisection genus is a homeomorphism invariant, i.e., the following conjecture is true (for a detailed discussion, see Remark~\ref{rem:addreltg}):

\begin{conj}\label{conj:exotic-tg}
 If two smooth $4$-manifolds $X$ and $Y$ are exotic, then they satisfy $g(X)=g(Y)$.
\end{conj}

%
%
%

Here, two smooth manifolds are called \textit{exotic} if they are homeomorphic but not diffeomorphic.
In \cite{LamMei20}, Lambert-Cole and Meier also mentioned that if Conjecture~\ref{conj:exotic-tg} is true, then there are no exotic smooth structures on the closed $4$-manifolds with trisection genus at most $2$ (e.g., $S^4$, $\CC P^2$, $S^1\times S^3$, $S^2\times S^2$, $\CC P^2\# \CC P^2$, and $\CC P^2\#\overline{\CC P^2}$ ; see \cite{MeiZup17_1}).
%
%

 The main purpose of this paper is to provide supporting evidence for Conjecture~\ref{conj:exotic-tg}.
 It was shown in \cite{MeiZup18} and \cite{LamMei20} that there exist infinitely many exotic pairs of closed $4$-manifolds with the same trisection genus.
 In the relative case, previously, there had been no known examples of exotic $4$-manifolds with non-spherical boundary whose trisection genera are determined.
 As far as the author knows, the smallest trisection genus of the known exotic pairs satisfying the condition of Conjecture~\ref{conj:exotic-tg} is $23$.
 (By the work of Spreer and Tillmann~\cite{SprTil18}, one can see that $K3 \# \overline{\CC P^2}$ and $3\CC P^2 \# 20\overline{\CC P^2}$ have the trisection genus $23$.)
 The following result gives an exotic pair with a much smaller trisection genus:

\begin{thm}\label{main:g=4exotic}
 There exists an exotic pair $(P_1, Q_1)$ of $4$-manifolds with boundary such that $g(P_1)=g(Q_1)=4$.
\end{thm}

\begin{figure}[!h]
	\begin{tabular}{cc}
		\begin{minipage}[t]{0.45\hsize}
		\centering
		\includegraphics[scale=0.9]{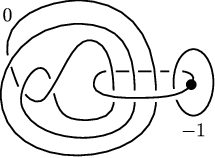}
		\caption{$P_{1}$}
		\label{fig:Kd-P1}
		\end{minipage} &
		\begin{minipage}[t]{0.45\hsize}
		\centering
		\includegraphics[scale=0.9]{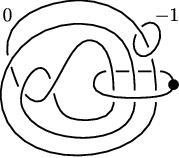}
		\caption{$Q_{1}$}
		\label{fig:Kd-Q1}
		\end{minipage}
	\end{tabular}
\end{figure}

 We prove this theorem by constructing explicit minimal genus relative trisection diagrams for the exotic pair $(P_1,Q_1)$ given by the handlebody diagrams of Figures~\ref{fig:Kd-P1} and \ref{fig:Kd-Q1}.
 It is well-known that $P_1$ and $Q_1$ are related by twisting along the \textit{Akbulut cork} (which is diffeomorphic to $W^-(0,0)$ given in Figure~\ref{fig:Kd-W-}).
 Let us recall that a \textit{cork} is a contractible $4$-manifold with an involution on the boundary, and is useful for constructing exotic smooth structures on $4$-manifolds. (For details, see Subsection~\ref{subsec:corks}.)
 In fact, our relative trisections of the exotic pair $(P_1, Q_1)$ are based on two distinct relative trisection diagrams of the Akbulut cork (see Remark~\ref{rem:2Acork}).

 Theorem~\ref{main:g=4exotic} is obtained as a byproduct of our construction of relative trisections for an infinite family of contractible $4$-manifolds.
 For integers $l$ and $k$, let $W^\pm(l,k)$ be the $4$-manifolds given by Figures~\ref{fig:Kd-W-} and \ref{fig:Kd-W+}. 
 The family $\{W^\pm(l,k)\}_{l,k\in\ZZ}$ was introduced by Akbulut and Kirby~\cite{AkbKir79} as examples of Mazur-type $4$-manifolds (i.e., contractible $4$-manifolds that admit a handle decomposition consisting of a single $0$-, $1$-, and $2$-handle).
 We note that $W^-(0,0)$ is diffeomorphic to the Akbulut cork.
 Our second result is the following:

\begin{thm}\label{main:AKMazur}
 For any integers $l$ and $k$, each of the Mazur-type $4$-manifolds $W^\pm(l,k)$ admits a genus-$3$ relative trisection.
 Moreover, if $l+k\notin\{2,3,4,5\}$, then the trisection genus of $W^-(l,k)$ is $3$.
\end{thm}

\begin{figure}[!h]
	\begin{tabular}{cc}
		\begin{minipage}[t]{0.45\hsize}
		\centering
		\includegraphics[scale=0.9]{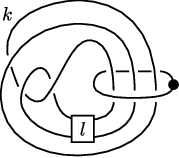}
		\caption{$W^-(l,k)$}
		\label{fig:Kd-W-}
		\end{minipage} &
		\begin{minipage}[t]{0.45\hsize}
		\centering
		\includegraphics[scale=0.9]{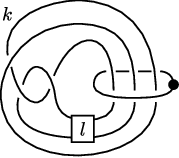}
		\caption{$W^+(l,k)$}
		\label{fig:Kd-W+}
		\end{minipage}
	\end{tabular}
\end{figure}

 Note that a genus-$g$ relative trisection refers to a relative trisection with the triple intersection surface of genus $g$.
 In Section~3, we give genus-$3$ relative trisection diagrams of $W^\pm(l,k)$ (see Figure~\ref{fig:td-D_n}).
 In fact, there are no known examples of contractible $4$-manifolds (other than the $4$-ball) that admit relative trisections of genus less than $3$.
 Thus, the following question naturally arises:

\begin{qst}~\label{qst:ctr-tg}
 Does there exist a contractible $4$-manifold $X$ with $g(X)<3$ other than the $4$-ball?
\end{qst}

 As we will see later, the contractible $4$-manifold $W^-(l,k)$ could potentially admit a genus-$2$ relative trisection if $l+k\in\{2,3,4,5\}$, however, we have not been able to prove the existence (or the non-existence) of such a structure.
 For more examples of trisected contractible $4$-manifolds, see \cite{Sen22} and \cite{Tak22a}.
 In \cite{Sen22}, {\c{S}}en produced a relative trisection of a Mazur-type $4$-manifold bounded by the Brieskorn sphere $\Sigma(2,5,7)$.
 Independently, the author~\cite{Tak22a} showed that there exist infinitely many corks with trisection genus $3$.

\section{Preliminaries}\label{sec:prel}

 We use the following conventions throughout this paper:
 All manifolds and surfaces are assumed to be compact, connected, oriented, and smooth.
 For a manifold $X$, we let $\overline{X}$ denote $X$ with the opposite orientation.
 In addition, if two smooth manifolds $X$ and $Y$ are orientation-preserving diffeomorphic, then we write $X\cong Y$.


\subsection{Relative trisections}

 The notion of a relative trisection was firstly introduced by Gay and Kirby~\cite{GayKir16} as a trisection for a compact, connected, oriented, smooth $4$-manifold with non-empty connected boundary.
 In this subsection, we introduce the fundamental properties of relative trisections.
 (See \cite[Section~6]{GayKir16} and \cite[Section~3]{CasGayPin18_1} for the precise definition of relative trisections.)

 Let $g$, $k$, $p$, and $b$ be integers satisfying the inequalities $g,k,p\geq0$, $b\geq1$, and $2p+b-1\leq k\leq g+p+b-1$.
 For a compact, connected, oriented, smooth $4$-manifold $X$ with non-empty connected boundary, if a decomposition $X=X_1\cup X_2\cup X_3$ is a $(g,k;p,b)$-relative trisection, then the following properties hold:
\begin{itemize}
\item
 Each sector $X_i$ is diffeomorphic to the $4$-dimensional $1$-handlebody of genus~$k$.
\item
 For each integer $i\in\{1,2,3\}$, taking indices mod $3$, the double intersection $X_i\cap X_{i+1}(=\partial{X_i}\cap\partial{X_{i+1}})$ is a compression body diffeomorphic to the $3$-dimensional $1$-handlebody of genus $g+p+b-1$.
\item
 The triple intersection $X_1\cap X_2\cap X_3$ is diffeomorphic to the genus-$g$ surface with $b$ boundary components.
\item
 There exists an open book decomposition of $\partial{X}$ with pages of genus $p$ with $b$ boundary components.
\end{itemize}

 As we mentioned in the introduction, a trisection naturally leads to an integer-valued invariant of smooth $4$-manifolds.
 The \textit{genus} of a trisection $X=X_1\cup X_2\cup X_3$ is the genus of the triple intersection surface $X_1\cap X_2\cap X_3$.
 We remark that this value is defined for a trisection rather than for a $4$-manifold.
 The trisection genus of a smooth $4$-manifold $X$, denoted by $g(X)$, is defined as the minimal integer $g$ such that $X$ admits a (relative) trisection of genus $g$.

\begin{rem}\label{rem:addreltg}
 Lambert-Cole and Meier showed in \cite[Proposition~1.7]{LamMei20} that if trisection genus for closed $4$-manifolds is additive under connected sum, then Conjecture~\ref{conj:exotic-tg} is true (i.e., exotic closed $4$-manifolds have the same trisection genus).
 Here, we give a proof that the same statement holds for the relative case.
 Assume that trisection genus for $4$-manifolds with boundary is additive under boundary connected sum.
 Let $(X,Y)$ be an exotic pair of $4$-manifolds with non-empty connected boundary.
 By the generalized Wall's theorem~\cite{Gom84} (see \cite{Wal64} for the original work of Wall), there exists a positive integer $n$ such that $X\#(\#^nS^2\times S^2)\cong Y\#(\#^nS^2\times S^2)$.
Note that a connected sum $M\#N$ with $\partial{M}\neq\emptyset$ and $\partial{N}=\emptyset$ is diffeomorphic to the boundary connected sum $M\natural(N-\Int{D^4})$ (see \cite[p.128]{GomSti99b}).
%
 Thus, we have $X\natural((\#^nS^2\times S^2)-\Int{D^4})\cong Y\natural((\#^nS^2\times S^2)-\Int{D^4})$.
 By the assumption, it follows that $g(X)=g(Y)$.
\end{rem}

 A relative trisection diagram is a description of a relatively trisected $4$-manifold.
 We now introduce the definition of relative trisection diagrams given by Castro, Gay, and Pinz\'{o}n-Caiced~\cite{CasGayPin18_1} .

\begin{defi}
 For an integer $i\in \{1,\ldots,n\}$, let $\alpha^i$ and $\beta^i$ be families of pairwise disjoint simple closed curves on surfaces $\Sigma$ and $\Sigma'$, respectively.
 The two $n+1$-tuples $(\Sigma;\alpha^1,\ldots,\alpha^n)$ and $(\Sigma';\beta^1,\ldots,\beta^n)$ are called \textit{diffeomorphism and handle slide equivalent}  if they are related by diffeomorphisms on $\Sigma$ and handle slides within each $\alpha^i$ (i.e., we are only allowed to slide curves from $\alpha^i$ over other curves from $\alpha^i$, but not over curves from $\alpha^j$ when $j\neq i$).
\end{defi}

\begin{defi}
 Let $\alpha$, $\beta$, and $\gamma$ be families of $g-p$ pairwise disjoint simple closed curves on a genus-$g$ surface $\Sigma$ with $b$ boundary components.
 A $4$-tuple $(\Sigma;\alpha,\beta,\gamma)$ is called a $(g,k;p,b)$-\textit{relative trisection diagram} if $(\Sigma;\alpha,\beta)$, $(\Sigma;\beta,\gamma)$, and $(\Sigma;\gamma,\alpha)$ are diffeomorphism and handle slide equivalent to the standard diagram $(\Sigma;\delta,\epsilon)$ of type $(g,k;p,b)$ shown in Figure \ref{fig:std-rtd}, where the red curves are $\delta$ and the blue curves are $\epsilon$.
\end{defi}
\begin{figure}[!h]
\centering
\includegraphics[scale=0.9]{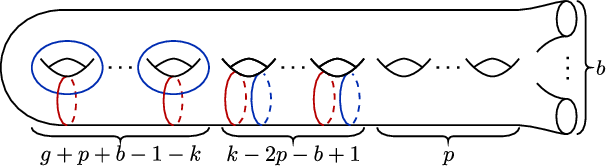}
\caption{The standard diagram $(\Sigma;\delta,\epsilon)$ of type $(g,k;p,b)$.}
\label{fig:std-rtd}
\end{figure}

 For a relative trisection diagram $(\Sigma;\alpha,\beta,\gamma)$, we draw $\alpha$, $\beta$, and $\gamma$ curves by red, blue, and green curves, respectively.
 See Figure~\ref{fig:td-D_n} for example.
 A pair of two black disks labeled with the same white number indicates an attaching of a cylinder. 
 More precisely, we remove these two black disks from the surface and glue in a cylinder $S^1\times [0,1]$ along the boundaries of these disks to obtain an oriented surface of one higher genus.

 The following theorem gives a correspondence between relative trisections and relative trisection diagrams:

\begin{thm}[Castro--Gay--Pinz\'{o}n-Caiced~{\cite{CasGayPin18_1}}]\label{thm:rt-rtd}
There is a natural bijection
\begin{align*}
\frac{\{\text{relative trisections}\}}{\text{diffeomorphism}} \to \frac{\{\text{relative trisection diagrams}\}}{\text{diffeomorphism and handleslide equivalent}}.
\end{align*}
\end{thm}

 We sometimes denote a relative trisection diagram by the symbol $\calD$.
 If $X$ is diffeomorphic to the trisected $4$-manifold corresponding to $\calD$ by Theorem~\ref{thm:rt-rtd}, then we simply say that $\calD$ is a relative trisection diagram of $X$.

\subsection{Corks and Mazur-type $4$-manifolds}\label{subsec:corks}

 We first review the definition and basic properties of corks.

\begin{defi}
 Let $C$ be a compact, contractible, smooth $4$-manifold with boundary and $\tau:\partial{C}\to \partial{C}$ be a smooth involution on the boundary.
 The pair $(C,\tau)$ is called a \textit{cork}, if $\tau$ extends to a self-homeomorphism of $C$, but cannot extend to any self-diffeomorphism of $C$.
\end{defi}

 A cut-and-paste operation along a cork is useful for constructing exotic smooth structures on $4$-manifolds.
 Suppose that a smooth $4$-manifold $X$ contains a cork $C$ as a submanifold. 
 Let $X'$ be the $4$-manifold obtained by removing $C$ from $X$ and re-gluing it by $\tau$ (i.e. $X' := (X-C)\cup_{\tau}C$).
 Then, $X$ and $X'$ are homeomorphic, but they may not be diffeomorphic. 
 This operation is called a \textit{cork twist} along $(C,\tau)$.
 An important fact is that, conversely, any two simply-connected, closed, exotic $4$-manifolds are related by a cork twist (\cite{Mat96}, \cite{CurFreHsiSto96}).

 The Mazur-type $4$-manifold $W^-(0,0)$ (see Figure~\ref{fig:Kd-W-}) admits a cork structure (\cite{Akb91_1}).
 It is the first example of a cork and is called the Akbulut cork.
 The following gives infinitely many exotic pairs obtained by twisting along the Akbulut cork:

\begin{exm}[Akbulut--Yasui~{\cite[Subsection~9.1]{AkbYas13}}, see also \cite{Akb91_1} for the case $n=1$]
 For a positive integer $n$, let $P_{n}$ and $Q_{n}$ be the $4$-manifolds given by the handlebody diagrams of Figures~\ref{fig:Kd-P_n} and \ref{fig:Kd-Q_n}, respectively.
 Then $(P_{n}, Q_{n})$ is an exotic pair, and is obtained by twisting along the Akbulut cork.  
\end{exm}

 We note that $P_n$ and $Q_n$ are related by exchanging the zero and the dot in the diagrams. 
 In this paper, we mainly focus on the exotic pair $(P_1, Q_1)$.
\begin{figure}[!h]
	\begin{tabular}{cc}
		\begin{minipage}[t]{0.45\hsize}
		\centering
		\includegraphics[scale=0.9]{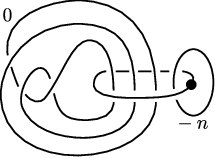}
		\caption{$P_{n}$}
		\label{fig:Kd-P_n}
		\end{minipage} &
		\begin{minipage}[t]{0.45\hsize}
		\centering
		\includegraphics[scale=0.9]{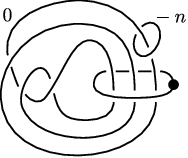}
		\caption{$Q_{n}$}
		\label{fig:Kd-Q_n}
		\end{minipage}
	\end{tabular}
\end{figure}

 Recall that a compact, contractible, smooth $4$-manifold $X$ with boundary is called Mazur-type, if $X$ admits a handle decomposition consisting of single $0$-, $1$-, and $2$-handle.
 In \cite{Maz61}, Mazur introduced the first example of such a $4$-manifold, which is diffeomorphic to $W^+(0,0)$.
 Many of the known corks are Mazur-type (see e.g., \cite{AkbMat98}, \cite{AkbYas08}, \cite{AukKimMelRub15}, \cite{DaiHedMal23}).
 The Mazur-type $4$-manifolds $W^\pm(l,k)$ given by Figures~\ref{fig:Kd-W-} and \ref{fig:Kd-W+} can be considered as a generalization of the Akbulut cork.
 Here, we introduce some properties of $W^\pm(l,k)$.

\begin{prop}[Akbulut--Kirby~{\cite[Proposition~1]{AkbKir79}}, see also {\cite{Akb16a}}]\label{prop:AK}
 For any integers $l$ and $k$, the following relations hold:
\begin{itemize}
\item $W^\pm(l,k)\cong W^\pm(l+1,k-1)$.
\item $W^-(l,k)\cong \overline{W^+(-l,-k+3)}$.
\end{itemize}
\end{prop}

 In addition to the above, it is known that the boundary of $W^+(0,k)$ is a Brieskorn homology sphere for some $k$ (e.g. $\partial{W^+(0,0)}\cong \Sigma(2,5,7)$, $\partial{W^+(-1,0)}\cong \Sigma(3,4,5)$, and $\partial{W^+(1,0)}\cong \Sigma(2,3,13)$).

\section{Relative trisections of $W^\pm(l,k)$.}

 In this section, we show that each of the Mazur-type $4$-manifolds $W^\pm(l,k)$ admits a genus-$3$ relative trisection.
 We first give a genus-$3$ relative trisection diagram, and then show that the induced $4$-manifold is diffeomorphic to $W^-(l,k)$.
 Our proof is similar to that of \cite[Theorem~1.3]{Tak22a}.

\begin{lem}\label{lem:3304td}
 For an integer $n$, let $\calD_n=(\Sigma;\alpha,\beta, \gamma)$ be the diagram shown in Figure~\ref{fig:td-D_n}.
 Then, $\calD_n$ is a $(3,3;0,4)$-relative trisection diagram.
\end{lem}

\begin{figure}[!htbp]
\centering
\includegraphics[scale=0.65]{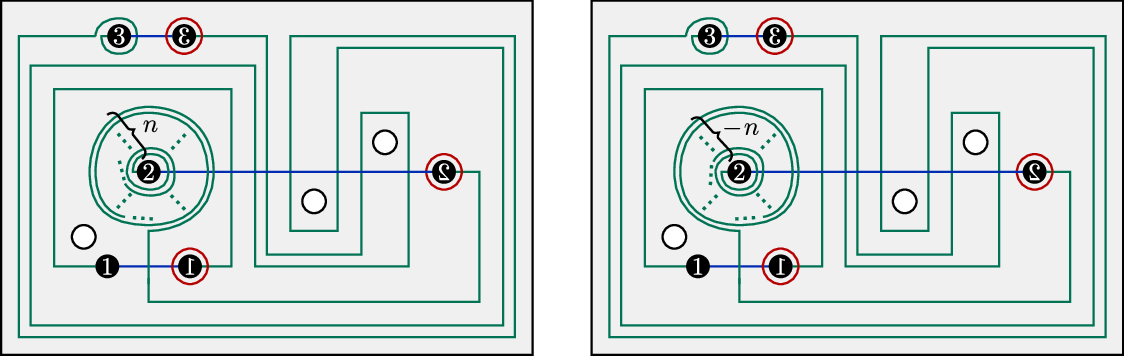}
\caption{Left: $\calD_n$ for $n>0$. Right: $\calD_n$ for $n\leq0$.}
\label{fig:td-D_n}
\end{figure}

\begin{proof}
 We prove that the $3$-tuples $(\Sigma,\alpha,\beta)$, $(\Sigma,\beta,\gamma)$, and $(\Sigma,\gamma,\alpha)$ are diffeomorphism and handle slide equivalent to the standard diagram of type $(3,3;0,4)$ shown in Figure~\ref{fig:std3304}.
 It is easy to see that $(\Sigma;\alpha,\beta)$ is standard.
 To prove for the remaining cases, we use the operations shown in Figure~\ref{fig:modif1234}, which were introduced in Section~3 of \cite{Tak22a}.
 The operations (i), (iii), and (iv) are obtained by Dehn twists, and (ii) is obtained by a handle slide of curves.

 In the following, we describe the proof for the case $n=-2$ (see Figures~\ref{fig:td-W^-(0,k)-Sigma-ca} and \ref{fig:td-W^-(0,k)-Sigma-bc}).
 We note that the same operations can be performed for any other integer $n$.
 By the diffeomorphisms on $\Sigma$ shown in Figure~\ref{fig:td-W^-(0,k)-Sigma-ca}, we can modify $(\Sigma;\gamma,\alpha)$ into the standard diagram.
 The resulting diagram is obtained by dragging the black disks labeled with ``$3$'', ``$2$'', and ``$1$'' along the marked $\gamma$ curves. 
 We remark that one can ignore the number of rotations of a $\gamma$ curve with respect to a black disk by using the operation (i).
 The proof of the case $(\Sigma;\beta,\gamma)$ is shown in Figure~\ref{fig:td-W^-(0,k)-Sigma-bc}.
 Each of the third and fifth diagrams is obtained by dragging the black disk with the blue circle along the marked $\gamma$ curves.
 In these processes, when the black disks approach $\beta$ curves, then it can pass through by using the operation (ii).
 Applying the operation (iv) to the last diagram, we obtain the standard diagram.
\end{proof}

\begin{figure}[!tbp]
\centering
\includegraphics[scale=0.8]{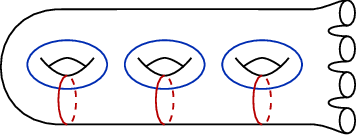}
\caption{The standard diagram of type $(3,3;0,4)$.}
\label{fig:std3304}
\end{figure}
\begin{figure}[!tbp]
\centering
\includegraphics[scale=0.85]{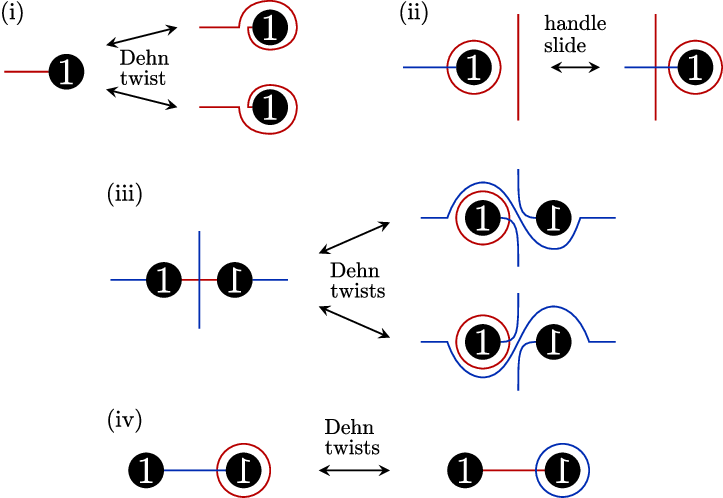}
\caption{The operations (i), (ii), (iii), and (iv).}
\label{fig:modif1234}
\end{figure}
\begin{figure}[!tbp]
\centering
\includegraphics[scale=0.85]{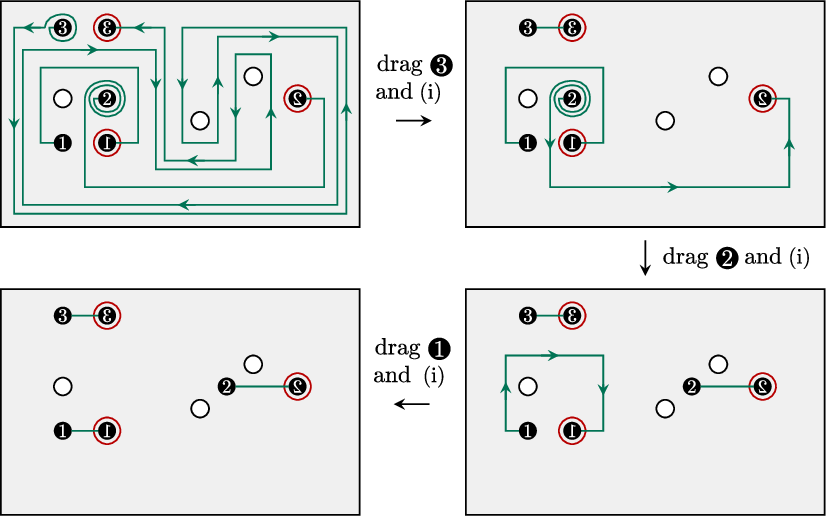}
\caption{Diffeomorphisms that modify $(\Sigma;\gamma,\alpha)$ into the standard diagram.}
\label{fig:td-W^-(0,k)-Sigma-ca}
\end{figure}
\begin{figure}[!tbp]
\centering
\includegraphics[scale=0.85]{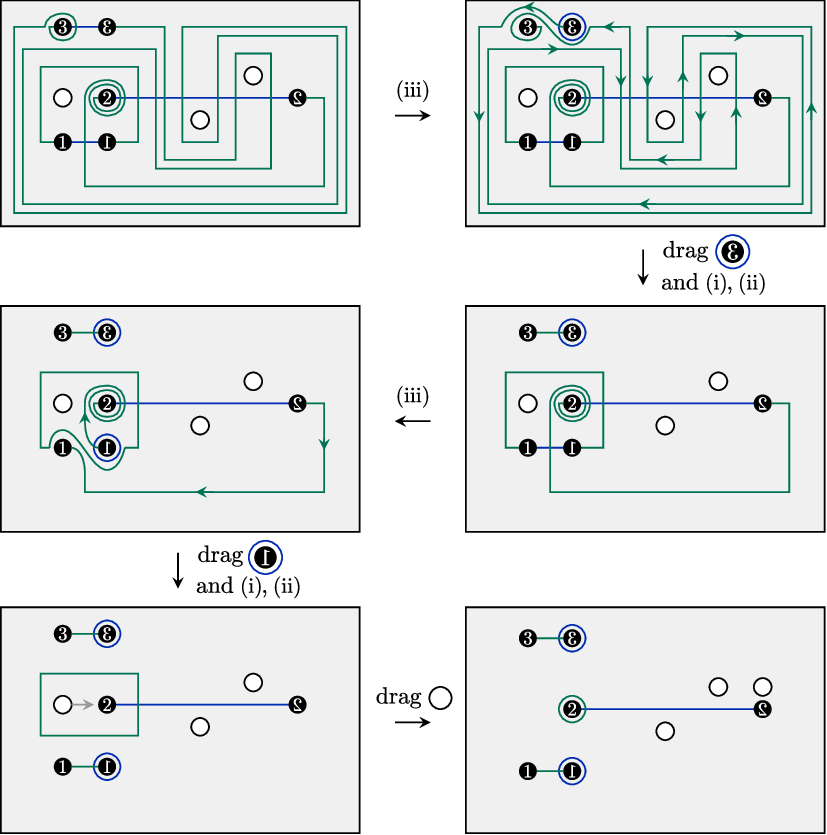}
\caption{Diffeomorphisms and handle slides that modify $(\Sigma;\beta,\gamma)$ into the standard diagram.}
\label{fig:td-W^-(0,k)-Sigma-bc}
\end{figure}

\begin{proof}[Proof of Theorem~\ref{main:AKMazur}] 
 We prove that, for any integers $l$ and $k$, each of the $4$-manifolds $W^\pm(l,k)$ is induced by the $(3,3;0,4)$-relative trisection diagram $\calD_n$ for some $n$.
 To prove this, we use an algorithm that produces a handlebody diagram from a relative trisection diagram.
This algorithm was introduced by Kim and Miller~\cite{KimMil20}.
%
The procedure for applying the algorithm to the relative trisection diagram $\calD_n$ is shown in Figure~\ref{fig:td-D_n-KMalg}.
\begin{figure}[!tbp]
\centering
\includegraphics[scale=0.75]{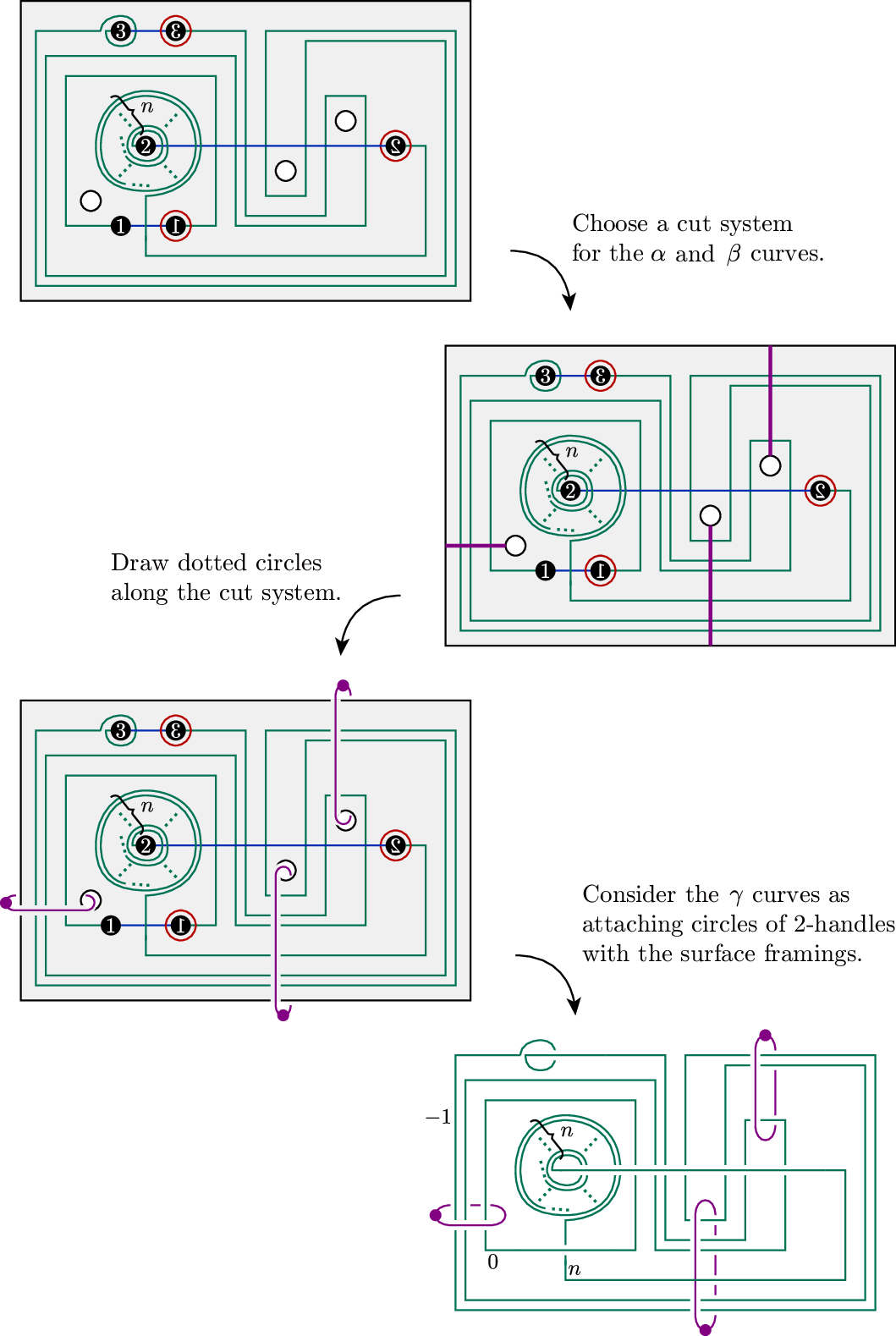}
\caption{An algorithm that produces a handlebody diagram from a relative trisection diagram.}
\label{fig:td-D_n-KMalg}
\end{figure}
(In this case, since $k=2p+b-1$, there are no 3-handles.)
First, we standardize the $\alpha$ and $\beta$ curves using handle slides and diffeomorphisms of $\Sigma$.
The first diagram $\calD_n$ in Figure~\ref{fig:td-D_n-KMalg} is already in this form.
Next, we draw $2p+b-1$ pairwise disjoint, properly embedded simple arcs on $\Sigma$ that are disjoint from $\alpha\cup \beta$ and cut $\Sigma_{\alpha}$ and $\Sigma_{\beta}$ into a disk. 
Such a family of arcs is called a \textit{cut system} for the $\alpha$ and $\beta$ curves.
Here, each of $\Sigma_\alpha$ and $\Sigma_\beta$ denotes the result of surgering $\Sigma$ along $\alpha$ and $\beta$, respectively.
%
%
%
%
Then, we draw dotted circles by doubling the arcs of the cut system (see the third diagram in Figure~\ref{fig:td-D_n-KMalg}).
%
%
Finally, we consider the $\gamma$ curves as attaching circles of $2$-handles. The framing of each attaching circle agrees with the surface framing.
By deleting the surface $\Sigma$ and the curves of $\alpha$ and $\beta$, we obtain the last diagram in Figure~\ref{fig:td-D_n-KMalg}.
%
%
%
%
%
The handle moves in Figure~\ref{fig:Kcalc-W^-(0,k)} shows that the $4$-manifold obtained from $\calD_n$ is diffeomorphic to $\overline{W^+(0,-n+1)}$.
By the second diffeomorphism of Proposition~\ref{prop:AK}, we see $\overline{W^+(0,-n+1)} \cong W^-(0,n+2)$.
Hence, $\calD_n$ is a relative trisection diagram of $W^-(0,n+2)$.
In other words, for any integer $n$, the $4$-manifold $W^-(0,n)$ is represented by the $(3,3;0,4)$-relative trisection diagram $\calD_{n-2}$.
\begin{figure}[!tbp]
\centering
\includegraphics[scale=0.9]{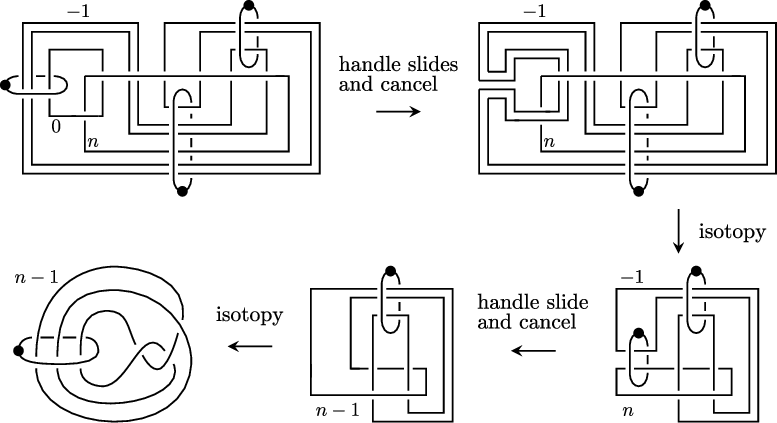}
\caption{Handle moves of the $4$-manifold induced by $\calD_n$.}
\label{fig:Kcalc-W^-(0,k)}
\end{figure}




We now prove that $W^-(l,k)\cong W^-(0,l+k)$ and $W^+(l,k)\cong \overline{W^-(0,-l-k+3)}$ hold for any integers $l$ and $k$.
%
%
By the first diffeomorphism of Proposition~\ref{prop:AK}, it follows that $W^\pm(l,k)\cong W^\pm(l+m,k-m)$ for any integer $m$. 
Thus, $W^\pm(0,l+k)\cong W^\pm(0+l,(l+k)-l)\cong W^\pm(l,k)$.
By the second diffeomorphism of Proposition~\ref{prop:AK}, we obtain $W^+(l',k')\cong \overline{W^-(-l',-k'+3)}$ for any integers $l'$ and $k'$.
Hence, $W^+(l,k)\cong W^+(0,l+k) \cong \overline{W^-(0,-l-k+3)}$ holds.
%
%
Therefore, for any integers $l$ and $k$, each of $W^\pm(l,k)$ is represented by the $(3,3;0,4)$-relative trisection diagram $\calD_{n}$ for some integer $n$.
%

 Next, we give a lower bound for the trisection genus of $W^-(l,k)$.
 By Corollary~4.2 of \cite{Tak22a}, if the boundary of a $4$-manifold $X$ is hyperbolic, then $g(X)\geq\chi(X)+2$.
 Comparing the complete list of integral exceptional surgeries along the Mazur link given by Yamada~\cite[Theorem~1.1]{Yam18}, we see that $\partial{W^-(0,k)}$ is hyperbolic if $k\notin\{2,3,4,5\}$.
 Since $W^-(0,k)$ is contractible, the inequality $g(W^-(0,k))\geq3$ holds for such $k$.
 Thus, if $l+k\notin\{2,3,4,5\}$, then $g(W^-(l,k))=g(W^+(-l,-k+3))=3$.
\end{proof}

\begin{rem}
 According to the list of \cite[Theorem~1.1]{Yam18}, the boundary of $W^-(l,k)$ is not hyperbolic if $l+k\in\{2,3,4,5\}$.
 In these cases, it holds that $g(W^-(l,k))\geq2$ by using the lower bound given in \cite[Corollary~4.2]{Tak22a}.
 That is, if $l+k\in\{2,3,4,5\}$, then $W^-(l,k)$ could admit a genus-$2$ relative trisection.
\end{rem}

 By the property that a relative trisection induces an open book decomposition on the boundary, we obtain the following corollary:

\begin{cor}
 For any integers $l$ and $k$, each of the homology $3$-sphere $\partial{W^\pm(l,k)}$ admits an open book decomposition with pages of a $4$-punctured $2$-sphere.
 Moreover, if $l+k\notin\{2,3,4,5\}$, then the minimal number of binding components of planar open book decompositions on $\partial{W^-(l,k)}$ is $4$.
\end{cor}

\begin{proof}
 Since $W^-(l,k)$ admits a $(3,3;0,4)$-relative trisection, there exists an open book decomposition with pages of genus-$0$ surface with $4$ boundary components.
 If a $3$-manifold admits a planar open book decomposition of the number of binding components less than $4$, then it is a Seifert fibered space or the connected sum of two lens spaces (see e.g., \cite{Ari08}).
 Recall that when $l+k\notin\{2,3,4,5\}$, the homology sphere $\partial{W^-(l,k)}$ is hyperbolic, and thus is irreducible and not Seifert fibered.
\end{proof}

\section{Trisection genera of exotic $4$-manifolds with boundary.}

 We give genus-$4$ relative trisections for the exotic pair $(P_1,Q_1)$ of $4$-manifolds with boundary given by Figures~\ref{fig:Kd-P1} and \ref{fig:Kd-Q1}.

\begin{lem}
 Let $\calD_{P_1}$ and $\calD_{Q_1}$ be diagrams shown in Figure~\ref{fig:td-PQ-cutarc}.
 (The purple arcs are elements of cut systems.)
 Then, they are $(4,3;0,4)$-relative trisection diagrams.
\end{lem}

\begin{figure}[!tbp]
\centering
\includegraphics[scale=0.87]{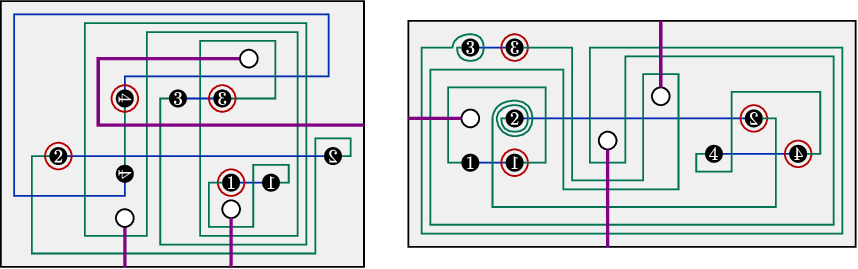}
\caption{Two $(4,3;0,4)$-relative trisection diagrams. The left diagram is $\calD_{P_1}$, and the right one is $\calD_{Q_1}$. The purple arcs are elements of cut systems.}
\label{fig:td-PQ-cutarc}
\end{figure}

\begin{proof}
 The proof is the same as that of Lemma~\ref{lem:3304td}.
 We verify that $(\Sigma,\alpha,\beta)$, $(\Sigma,\beta,\gamma)$, and $(\Sigma,\gamma,\alpha)$ of $\calD_{P_1}$ and $\calD_{Q_1}$ are diffeomorphism and handle slide equivalent to the standard diagram of type $(4,3;0,4)$.
 We omit proofs of easy parts $(\Sigma,\alpha,\beta)$ and $(\Sigma,\gamma,\alpha)$.
 The proof for the case $(\Sigma;\beta,\gamma)$ of $\calD_{P_1}$ and $\calD_{Q_1}$ are shown in Figures~\ref{fig:td-P-Sigma-bc} and \ref{fig:td-Q-Sigma-bc}, respectively.
\end{proof}

\begin{figure}[!tbp]
\centering
\includegraphics[scale=0.85]{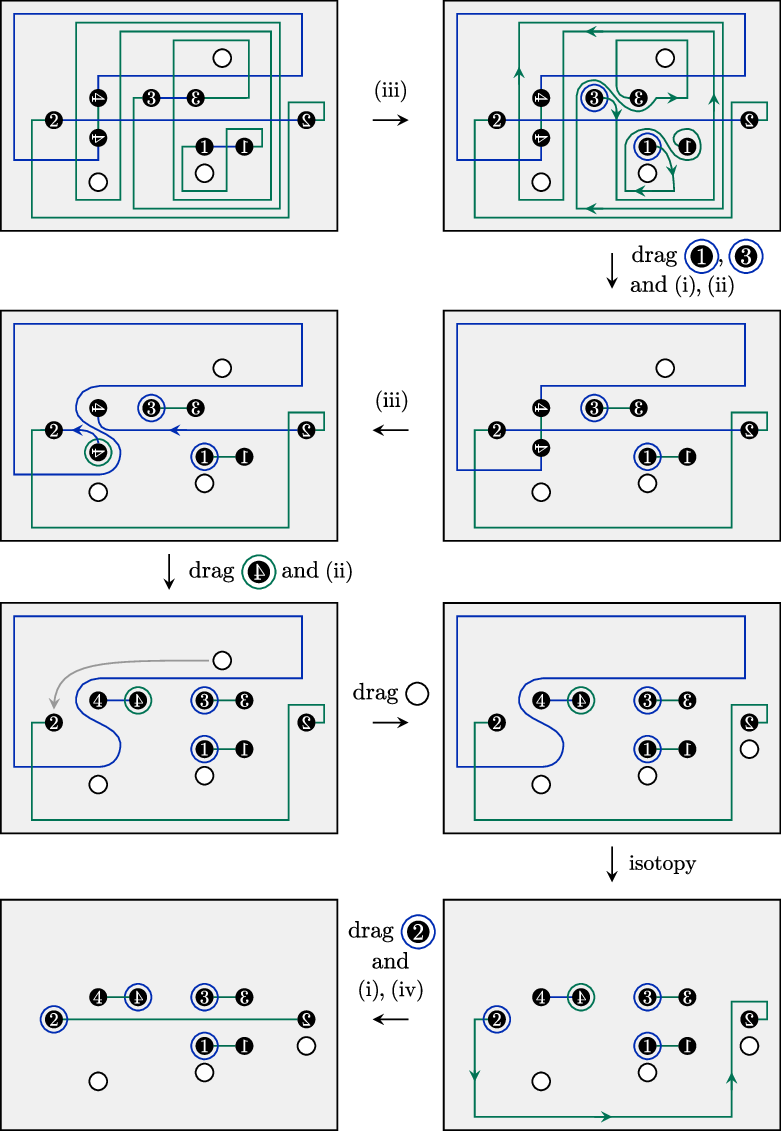}
\caption{Diffeomorphisms and handle slides that modify $(\Sigma;\beta,\gamma)$ of $\calD_{P_1}$ into the standard diagram.}
\label{fig:td-P-Sigma-bc}
\end{figure}
\begin{figure}[!tbp]
\centering
\includegraphics[scale=0.85]{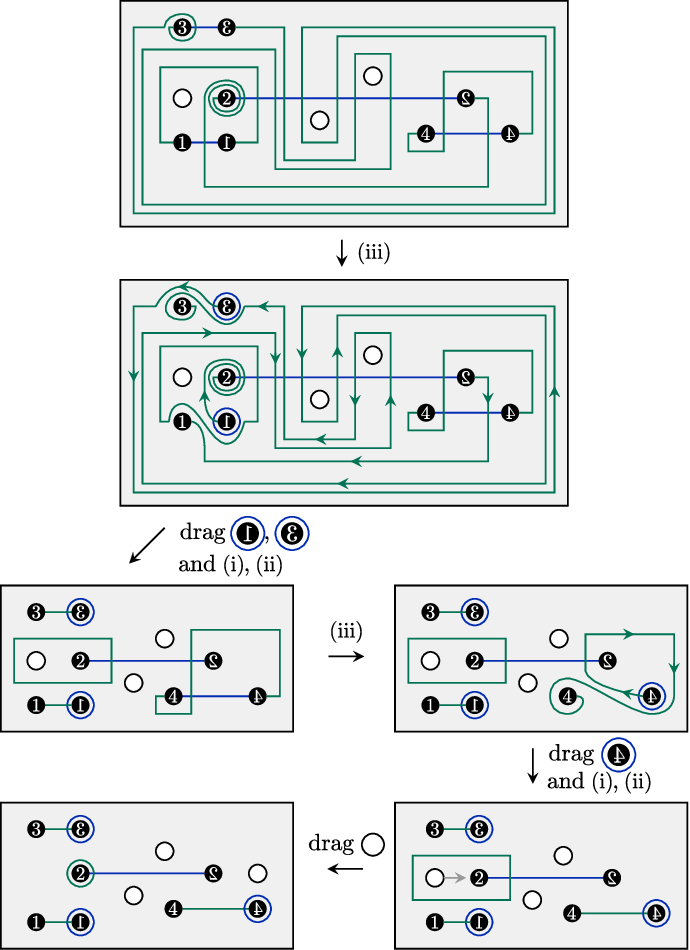}
\caption{Diffeomorphisms and handle slides that modify $(\Sigma;\beta,\gamma)$ of $\calD_{Q_1}$ into the standard diagram.}
\label{fig:td-Q-Sigma-bc}
\end{figure}

\begin{proof}[Proof of Theorem~\ref{main:g=4exotic}]
We show that the $4$-manifolds $P_1$ and $Q_1$ are induced by the $(4,3;0,4)$-relative trisection diagrams $\calD_{P_1}$ and $\calD_
{Q_1}$, respectively.
%
%
The relative trisection diagram $\calD_{P_1}$ induces the first handlebody diagram of Figure~\ref{fig:Kcalc-P} by using the algorithm described in Figure~\ref{fig:td-D_n-KMalg}.
%
%
%
The handle moves in Figure~\ref{fig:Kcalc-P} shows that the $4$-manifold induced by $\calD_{P_1}$ is diffeomorphic to $P_1$.
%
The operation $(*)$ is obtained by the procedure shown in Figure~\ref{fig:AKmove}. This operation was introduced by Akbulut and Kirby~\cite{AkbKir79}.
For the details of the final isotopy in Figure~\ref{fig:Kcalc-P}, see Figure~\ref{fig:Kcalc-isotopy}.
%
%
%
The relative trisection diagram $\calD_{Q_1}$ induces the handlebody diagrams of Figure~\ref{fig:Kd-Q2}.
In addition, we see that the $4$-manifold given by Figure~\ref{fig:Kd-Q2} is diffeomorphic to $Q_1$ by performing the same handle moves of Figure~\ref{fig:Kcalc-W^-(0,k)}.
 Hence, both $P_1$ and $Q_1$ admit $(4,3;0,4)$-relative trisections.

 Next, we give lower bounds of the trisection genera of $P_1$ and $Q_1$.
 For the handlebody diagram of them shown in Figure~\ref{fig:Kd-Q1}, by using the slam-dunk move, we see that $\partial{P_1}$ and $\partial{Q_1}$ are homeomorphic to $\partial{W^-(0,1)}$, which is hyperbolic (see \cite[Theorem 1.1]{Yam18}).
 By using the inequality of \cite[Corollary~4.2]{Tak22a}, we conclude that $g(P_1)=g(Q_1)\geq\chi(P_1)+2=4$.
\end{proof}

\begin{figure}[!tbp]
\centering
\includegraphics[scale=0.7]{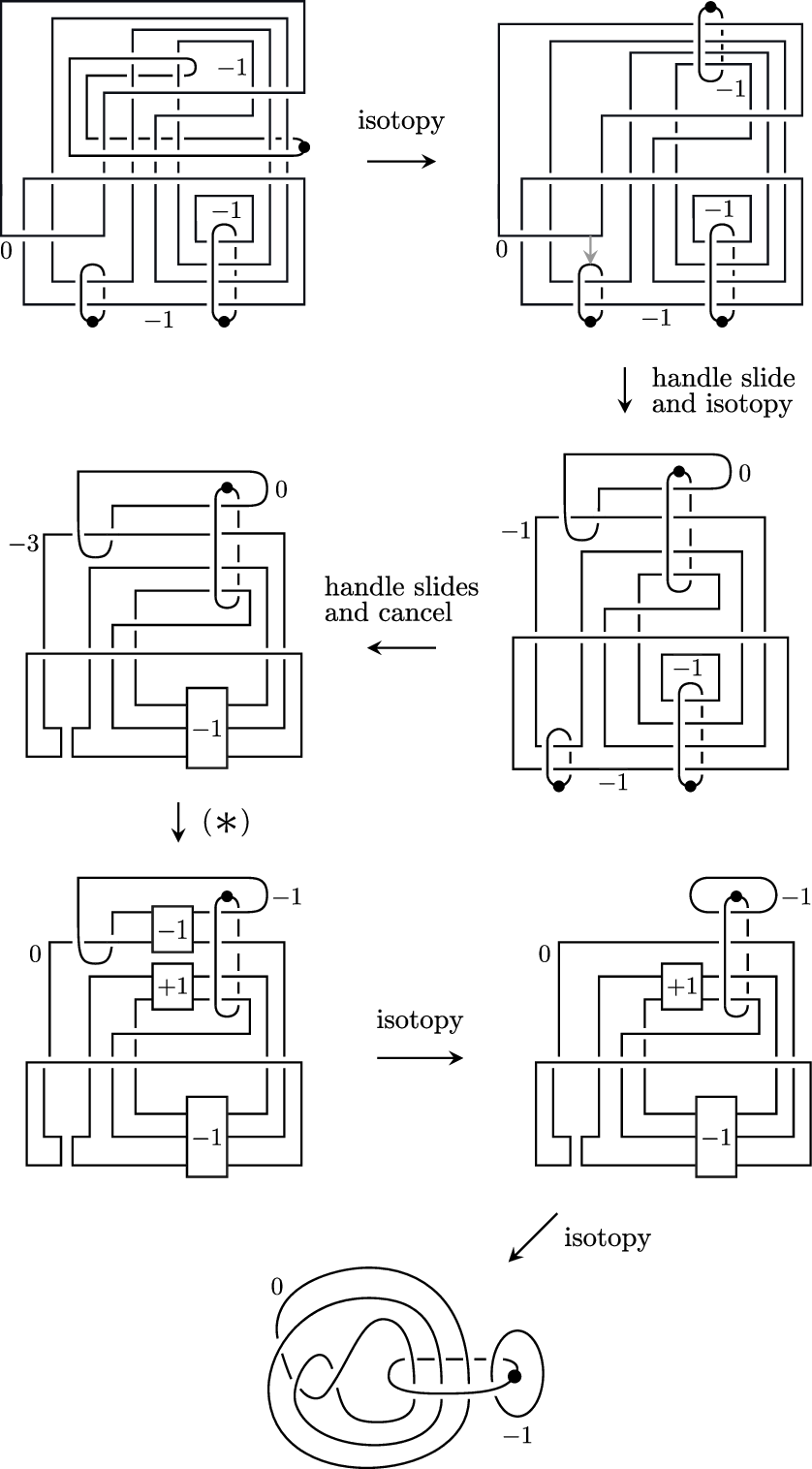}
\caption{Handle moves of the $4$-manifold induced by $\calD_{P_1}$.}
\label{fig:Kcalc-P}
\end{figure}
\begin{figure}[!tbp]
\centering
\includegraphics[scale=0.7]{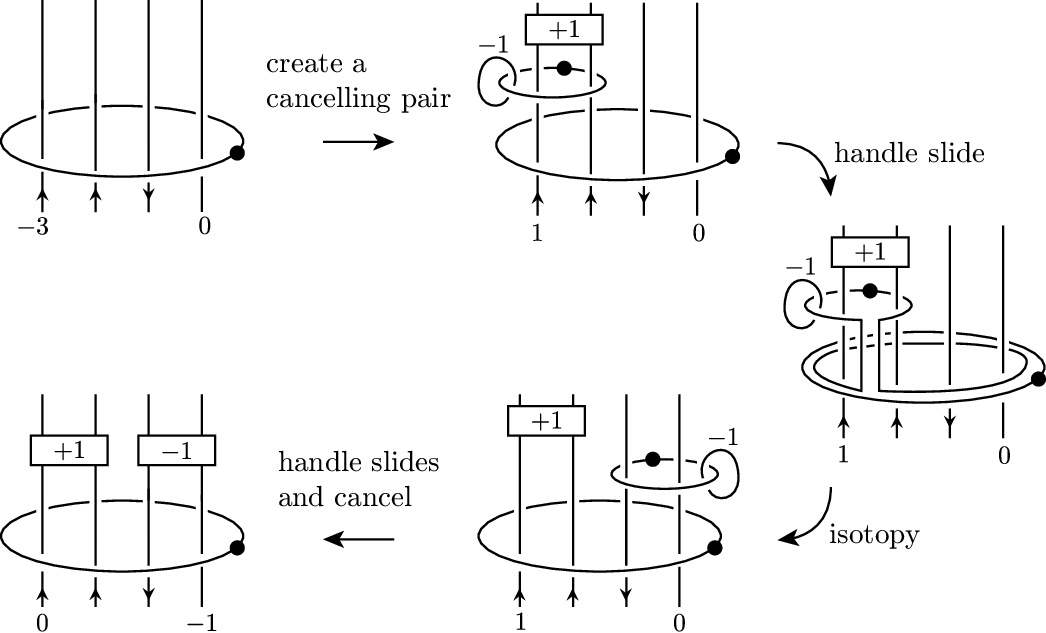}
\caption{The detailed procedure of the operation $(*)$.}
\label{fig:AKmove}
\end{figure}
\begin{figure}[!tbp]
\centering
\includegraphics[scale=0.7]{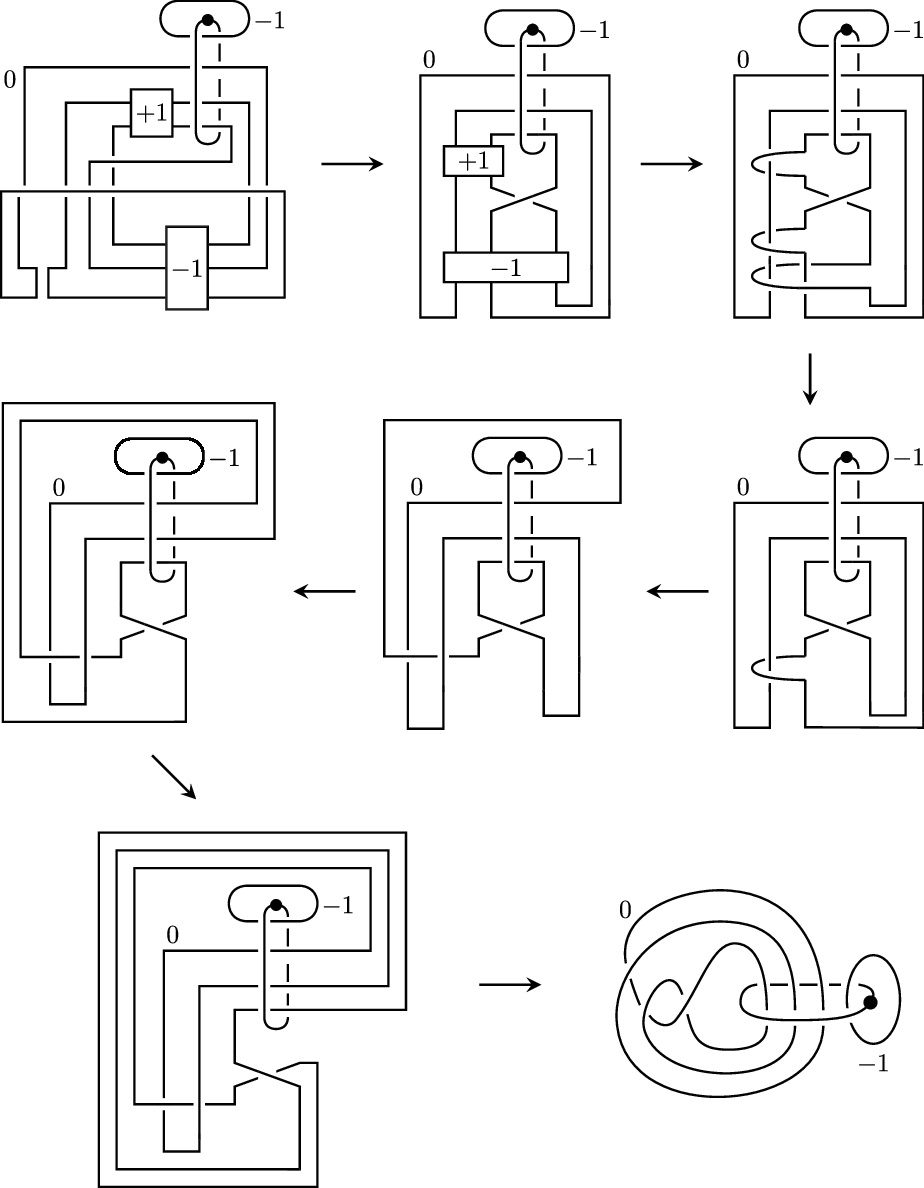}
\caption{The detailed procedure of the final isotopy in Figure~\ref{fig:Kcalc-P}}
\label{fig:Kcalc-isotopy}
\end{figure}
\begin{figure}[!tbp]
\centering
\includegraphics[scale=0.9]{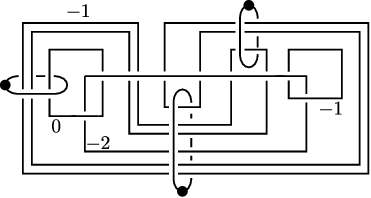}
\caption{The handlebody diagram induced by $\calD_{Q_1}$.}
\label{fig:Kd-Q2}
\end{figure}

\begin{rem}\label{rem:2Acork}
 As we saw in the proof of Theorem~\ref{main:AKMazur}, the $(3,3;0,4)$-relative trisection diagram $\calD_n$ gives the contractible $4$-manifold $W^-(0,n+2)$.
 Thus, the Akbulut cork $W^-(0,0)$ is induced by $\calD_{-2}$ shown in the right of Figure~\ref{fig:td-2Acork}.
 We see that the $(4,3;0,4)$-relative trisection diagram $\calD_{Q_1}$ is obtained by modifying a part of $\calD_{-2}$.
 On the other hand, in \cite{Tak22a}, the author gave another $(3,3;0,4)$-relative trisection diagram of $W^-(0,0)$ (see the left of Figure~\ref{fig:td-2Acork}), which is very similar to $\calD_{P_1}$.
 That is, our relative trisections of the exotic pair $(P_1, Q_1)$ are based on two distinct relative trisection diagrams of the Akbulut cork.
 (Compare Figure~\ref{fig:td-PQ-cutarc} with Figure~\ref{fig:td-2Acork}.)
\end{rem}

\begin{figure}[!tbp]
\centering
\includegraphics[scale=0.9]{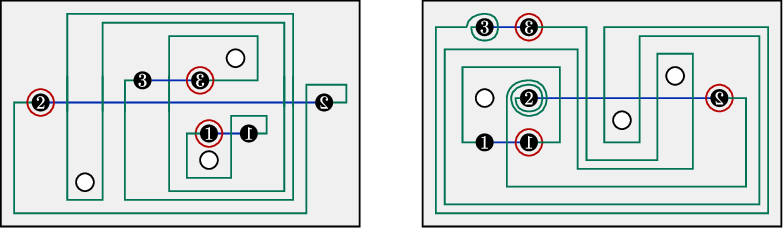}
\caption{Two $(3,3,0,4)$-relative trisection diagrams of the Akbulut cork $W^-(0,0)$. The left diagram is given in \cite[Figure~31]{Tak22a}, and the right one is $\calD_{-2}$.}
\label{fig:td-2Acork}
\end{figure}

\begin{rem}
 In \cite[Theorem~1.6]{Tak22a}, the author constructed small genus relative trisections for the exotic pair $(P_2, Q_2)$, and showed that $g(Q_2)=4$ and $g(P_2)=4$ or $5$.
 However, we have not yet been able to find a genus-$4$ relative trisection of $P_2$.
 If we have $g(P_2)=5$, then it follows that Conjecture~\ref{conj:exotic-tg} is false, namely trisection genus of $4$-manifolds with boundary is not a homeomorphism invariant.
\end{rem}

 Looking at the handlebody diagrams of Figures~\ref{fig:Kd-P1} and \ref{fig:Kd-Q1},  it seems that $(P_1,Q_1)$ is very simple exotic pairs of $4$-manifolds with boundary. Thus, we raise the following natural question:

\begin{qst}
 Does there exist an exotic pair $(X,Y)$ of $4$-manifolds such that $g(X)=g(Y)<4$?
\end{qst}

\subsection*{Acknowledgements}
 The author would like to express his adviser Kouichi Yasui for helpful comments and encouragement.
 The author also thanks Yohei Wakamaki for many discussions.
 The author was partially supported by JST SPRING, Grant Number JPMJSP2138.
%

\end{document}